\documentclass[a4paper,reqno,11pt]{amsart}

\usepackage{amssymb}
\usepackage{amstext}
\usepackage{amsmath}
\usepackage{amscd}
\usepackage{amsthm}
\usepackage{amsfonts}
\usepackage{array}
\usepackage{enumerate}
\usepackage{graphicx}
\usepackage{color}
\usepackage{here} 
\usepackage{bm} 
\usepackage{mathrsfs}
\usepackage{mathtools}
\usepackage{latexsym}
\usepackage{booktabs}
\usepackage{fancyhdr}
\usepackage{subcaption}
\usepackage{tikz}
\usepackage{tikz-cd}
\usetikzlibrary{arrows}
\usetikzlibrary{decorations.markings}
\usetikzlibrary{patterns,decorations.pathreplacing}

\makeatletter
  
\@addtoreset{equation}{section}
\makeatother

\newcommand{\calC}{\mathcal{C}}
\newcommand{\calD}{\mathcal{D}}
\newcommand{\calE}{\mathcal{E}}
\newcommand{\calF}{\mathcal{F}}
\newcommand{\calG}{\mathcal{G}}

\newcommand{\calS}{\mathcal{S}}
\newcommand{\calT}{\mathcal{T}}

\newcommand{\ZZ}{\mathbb{Z}}

\newcommand{\RR}{\mathbb{R}}

\newcommand{\TT}{\mathbb{T}}
\newcommand{\kk}{\Bbbk}

\newcommand{\scrC}{\mathscr{C}}

\newcommand{\ab}{\mathbf{a}}
\newcommand{\bfa}{\mathbf{a}}

\newcommand{\cb}{\mathbf{c}}

\newcommand{\kb}{\mathbf{k}}
\newcommand{\lb}{\mathbf{l}}

\newcommand{\bfx}{\mathbf{x}}
\newcommand{\xb}{\mathbf{x}}
\newcommand{\yb}{\mathbf{y}}

\newcommand{\zb}{\mathbf{z}}
\newcommand{\sfM}{\mathsf{M}}
\newcommand{\sfN}{\mathsf{N}}

\newcommand{\fkm}{\mathfrak{m}}

\newcommand{\Hom}{\operatorname{Hom}}

\newcommand{\conv}{\mathrm{conv}}

\newcommand{\cone}{\mathrm{cone}}

\newcommand{\GL}{\mathrm{GL}}

\newcommand{\Vol}{\mathrm{Vol}}

\newcommand{\Stab}{\mathrm{Stab}}

\newcommand{\set}[1]{\left\{ #1 \right\}}

\newcommand{\rbra}[1]{\left( #1 \right)}

\def\opn#1#2{\def#1{\operatorname{#2}}} 
\opn\Cl{Cl} \opn\conv{conv} \opn\deg{deg} \opn\rank{rank} \opn\Spec{Spec} \opn\Stab{Stab} \opn\aff{aff} \opn\div{div} \opn\GL{GL}
\opn\cone{cone} \opn\End{End} \opn\Hom{Hom} \opn\mod{mod} \opn\gldim{gldim} \opn\pdim{pdim} \opn\diag{diag} \opn\vert{vert}
\opn\Block{Block} \opn\Pyr{Pyr} \opn\max{max} \opn\min{min} \opn\ini{in} \opn\rev{rev} \opn\ker{ker} \opn\lat{lat} \opn\pull{pull} \opn\rev{rev}
\opn\Gale{Gale} \opn\sign{sign} \opn\supp{supp}
\opn\cok{coker} \opn\core{core} \opn\star{star}
\opn\pd{pd} \opn\Soc{Soc} \opn\Ap{Ap} \opn\Sym{Sym}
\opn\PF{PF} \opn\t{t} \opn\F{F} \opn\e{e}
\opn\m{m} \opn\G{G} \opn\g{g} \opn\H{H}
\opn \embdim{embdim} \opn\var{var}
\opn{\mult}{mult} \opn{\emb}{emb} \opn{\charac}{char}

\newtheorem{thm}{Theorem}[section]
\newtheorem{theorem}[thm]{Theorem}
\newtheorem{cor}[thm]{Corollary}

\newtheorem{lemma}[thm]{Lemma}
\newtheorem{prop}[thm]{Proposition}

\newtheorem{conj}[thm]{Conjecture}

\theoremstyle{definition}
\newtheorem{defi}[thm]{Definition}

\newtheorem{ex}[thm]{Example}

\theoremstyle{remark}

\begin{document}

\title{Dual $F$-signatures of Veronese subrings and Segre products of polynomial rings}
\author{Koji Matsushita}

\address{Department of Pure and Applied Mathematics, Graduate School of Information Science and Technology, Osaka University, Suita, Osaka 565-0871, Japan}
\email{k-matsushita@ist.osaka-u.ac.jp}

\subjclass[2020]{Primary 13A35, 05E40; Secondary 52B20, 14M25}

\keywords{Dual $F$-signatures, Toric rings, Veronese subrings of polynomial rings, Segre products of polynomial rings}

\maketitle

\begin{abstract} 
In this paper, we compute the dual $F$-signatures of certain toric rings by using combinatorial techniques.
Specifically, we calculate the dual $F$-signatures of Veronese subrings of polynomial rings.
Moreover, we give an upper bound for the dual $F$-signatures of Segre products of polynomial rings and show that this upper bound is attained in some cases.
\end{abstract}

\section{Introduction}
Recently, numerical invariants associated with commutative rings with positive characteristic have been defined and investigated, such as Hilbert-Kunz multiplicities (\cite{Kunz76,Mon83}), $F$-signatures (\cite{HunLeu,Tuc12}), dual $F$-signatures (\cite{San}) and so on.
These invariants tell us valuable information about the underlying ring $R$.
In fact, they characterize various properties of $R$, for example, regularity, strong $F$-regularity, $F$-rationality and Gorensteinness.
On the other hand, it is not easy to compute their values in concrete terms.
In this paper, we focus on the dual $F$-signatures of toric rings and compute them specifically in certain classes.

Let ($R$, $\fkm$) be a Noetherian local ring of prime characteristic $p>0$ and assume that the residue field $R/\fkm$ is algebraically closed.
In this situation, we can define the $e$-times iterated Frobenius morphism $F^e:R\to R\;(r\mapsto r^{p^e})$ for $e\in\ZZ_{>0}$. 
For an $R$-module $M$, let $F^e_{*}M$ denote the \textit{Frobenius push-forward} of $M$, which is an $R$-module given by restriction of scalars under $F^e$ (that is, $F^e_*M$ is just $M$ as an abelian group and its $R$-module structure is defined by $r\cdot m\coloneqq F^e(r)m=r^{p^e}m$ $\;(r\in R,\;m\in M)$).
We say $R$ is \textit{$F$-finite} if $F^e_*R$ is a finitely generated $R$-module. 
In this paper, we always assume that $R$ is $F$-finite since we only discuss such rings.
In addition, we assume that $R$ is reduced, and hence we can identify $F^e_*R$ with $R^{1/p^e}$, which is the $R$-algebra consisting of $p^e$th roots of $R$ inside an algebraic closure of its fraction field.

In \cite{San}, Sannai introduced the following value for an $R$-module $M$:
\begin{align}\label{dualM}
s(M) := \limsup_{e \to \infty} 
\frac{\max \{N : \text{ there is a surjection $F_*^e M \twoheadrightarrow M^{\oplus N}$}\}}
{\rank F_*^e M}.
\end{align}
The value $s(R)$ is called the \textit{$F$-signature} of $R$, which can also be described as the following form:
\begin{align}\label{flimi}
s(R)=\lim_{e\to \infty}\frac{\max \{N : R^{1/p^e} \cong R^{\oplus N}\oplus L\}}{\rank R^{1/p^e}},
\end{align}
where $L$ is an $R$-module without free direct summands (this is the original form of the definition of $F$-signatures (see \cite{HunLeu}) and the existence of the limit (\ref{flimi}) is shown in \cite{Tuc12}).

Sannai calls the limit (\ref{dualM}) the \textit{dual $F$-signature of $M$}.
On the other hand, in \cite{ST}, Smirnov and Tucker call $s(\omega_R)$ the \textit{dual $F$-signature of $R$} and write it by $s_{\mathrm{dual}}(R)$ where $\omega_R$ denotes the canonical module of $R$.
They also introduce the notion of the \textit{relative $F$-rational signature} $s_{\mathrm{rel}}(R)$ of $R$ and show that $s_{\mathrm{rel}}(R)$ coincides with $s(\omega_R)$ (\cite[Corollary~5.6 and Theorem~5.10]{ST}).

In this paper, we focus only on $s(\omega_R)$, thus we use the later definition, that is,
\begin{defi}
We define the \textit{dual $F$-signature} of $R$ as
$$s_{\mathrm{dual}}(R):=s(\omega_R)=\limsup_{e \to \infty} 
\frac{\max \{N : \text{ there is a surjection $F_*^e \omega_R \twoheadrightarrow \omega_R^{\oplus N}$}\}}
{\rank F_*^e \omega_R}.$$
\end{defi}

Regarding $s(R)$ and $s_{\mathrm{dual}}(R)$, the following facts are known:

\begin{thm}[{see \cite{AL,HunLeu,San,ST,WY}}]
Let $(R,\fkm)$ be a reduced $F$-finite Cohen--Macaulay local ring with $\charac\,R=p>0$ and assume that $R/\fkm$ is an algebraically closed field. 
Then we have the following.
\begin{itemize}
\item[(1)] $0\le s(R) \le s_{\mathrm{dual}}(R)\le 1$.
\item[(2)] The following are equivalent:
\begin{itemize}
    \item[(a)] $R$ is regular;
    \item[(b)] $s(R)=1$;
    \item[(c)] $s_{\mathrm{dual}}(R)=1$.
\end{itemize}
\item[(3)] $R$ is strongly $F$-regular if and only if $s(R)>0$.
\item[(4)] $R$ is $F$-rational if and only if $s_{\mathrm{dual}}(R)>0$.
\item[(5)] $R$ is Gorenstein if and only if $s(R)=s_{\mathrm{dual}}(R)$.
\end{itemize}
\end{thm}

This theorem shows that $s(R)$ and $s_{\mathrm{dual}}(R)$ measure the severity of singularities of $R$. 
Therefore, it is important to determine these invariants and consider what the explicit value of them means.

$F$-signatures have been given for several classes of commutative rings with positive characteristic (see, e.g., \cite{Bru,HN2,HunLeu,Sin,WY}).
In particular, it is known that the $F$-signature of a toric ring is equal to the volume of a certain convex polytope (\cite{Bru}).

On the other hand, as far as the author knows, there are very few classes for which the dual $F$-signature has been explicitly calculated (see \cite{N15,N18,San}).
Recently, a method for the computation of the dual 
$F$-signatures of toric rings has been developed.
In particular, Smirnov and Tucker give a formula to calculate them (\cite{ST}, see Proposition~\ref{prop go toric}).
However, using this formula, we have only been able to speculate on the values of the dual $F$-signatures of Veronese subrings of polynomial rings (\cite[Question~6.4]{ST}).

\bigskip

In this paper, we present new examples where the dual $F$-signatures are computed. 
We calculate them using Smirnov and Tucker's formula and an inequality associated with $O$-sequences.

Throughout this paper, we suppose that $\kk$ is an algebraically closed field of prime characteristic $p>0$.
First, we completely determine the dual $F$-signatures of Veronese subrings, i.e., we give the positive answer to \cite[Question~6.4]{ST}.

\begin{thm}[{Theorem~\ref{dual_ver}}]\label{main1}
Let $V_{n,d}$ be the $n$th Veronese subring of the polynomial ring over $\kk$ with $d$ variables.
Then we have 
$$s_{\mathrm{dual}}(V_{n,d})=\frac{1}{d}\left\lceil\frac{d}{n} \right\rceil.$$
\end{thm}

Moreover, we give an upper bound for the dual $F$-signatures of Segre products of polynomial rings by using their generalized $F$-signatures (we will introduce the notion of generalized $F$-signatures in Section~\ref{subsec_signature}).

\begin{thm}[{Theorem~\ref{dual_segre}}]\label{main2}
Let $S_{r_1,\ldots,r_t}:=\#_{i=1}^t\kk[x_{i,1},\ldots,x_{i,r_i+1}]$ with $r_1\le \cdots \le r_t$. Then we have
\begin{align}\label{ineq_dual}
s_{\mathrm{dual}}(S_{r_1,\ldots,r_t})\le \sum_{\zb\in \calD}\rbra{\prod_{i=1}^{t-1}\frac{\binom{r_t+z_i}{r_i}}{\binom{r_t}{r_i}}}s_{\mathrm{gen}}(M_{\zb}),
\end{align}
where 
$$\calD=\set{(y_1,\cdots,y_{t-1}) \in \ZZ^{t-1} : \hspace{-0.4cm}
\begin{array}{ll}
&-r_t+r_i \leq y_i \leq 0 \text{ for } i\in \set{1,\ldots,t-1}, \vspace{0.2cm}\\
&-r_j \leq y_i-y_j \leq r_i \text{ for }i,j\in \set{1,\ldots,t-1}
\end{array}}
$$
and $s_{\mathrm{gen}}(M_\zb)$ denotes the generalized $F$-signature of the conic divisorial ideal $M_\zb$ corresponding to $\zb=(z_1,\ldots,z_{t-1})\in \calD$.
\end{thm}
Actually, in the case of Segre products of polynomial rings, the generalized $F$-signatures are calculated by counting the elements in the symmetric group satisfying certain conditions (see \cite[Theorem~5.3]{HN2}).

In addition, we show that the equality of (\ref{ineq_dual}) holds when either of the following conditions holds (Propositions~\ref{attain1} and \ref{attain2}):
\begin{itemize}
    \item[(i)] $r_2=\cdots=r_t$.
    \item[(ii)] $r_i\in \set{r_1,r_1+1}$ for $i=2,\ldots,t$.
\end{itemize}
In particular, we have the following corollary:
\begin{cor}[{Proposition~\ref{attain1}}]\label{main3}
Let $S_{r_1,r_2}:=\kk[x_{1,1},\ldots,x_{1,r_1+1}]\#\kk[x_{2,1},\ldots,x_{2,r_2+1}]$ with $r_1\le r_2$, and let $d=r_1+r_2+1$.
Then we have
$$s_{\mathrm{dual}}(S_{r_1,r_2})=\frac{\sum_{l=r_1}^{r_2}\binom{l}{r_1}A_{l,d}}{\binom{r_2}{r_1}d!},$$
\end{cor}
\noindent where $A_{l,d}$ denotes the Eulerian number.

\bigskip

The structure of this paper is as follows.
In Section~\ref{sec_preli}, we recall the definitions and notation of toric rings and a method to compute the (generalized) $F$-signatures and dual $F$-signatures of toric rings.
We also recall the notion of $O$-sequences and give a lemma to calculate the dual $F$-signatures of our toric rings.
In Section~\ref{sec_vero}, we give the dual $F$-signatures of Veronese subrings of polynomial rings (Theorem~\ref{dual_ver}).
In Section~\ref{sec_segre}, we discuss the dual $F$-signatures of Segre products of polynomial rings.
We recall some properties on Segre products of polynomial rings and their generalized $F$-signatures.
We then provide an upper bound for the dual $F$-signatures of Segre products of polynomial rings (Theorem~\ref{dual_segre}) and show that this upper bound is achieved in some cases (Propositions~\ref{attain1} and \ref{attain2}).

\bigskip

\subsection*{Acknowledgement} 
The author would like to thank Akihiro Higashitani for his helpful comments and advice on improving this paper.
The author is partially supported by Grant-in-Aid for JSPS Fellows Grant JP22J20033.

\bigskip


\section{Preliminaries}\label{sec_preli}
The goal of this section is to prepare the required materials for the discussions of our main results.

\subsection{Toric rings and conic divisorial ideals}\label{subsec_preli}
In this paper, we will discuss Veronese subrings and Segre products of polynomial rings, which can be realized as toric rings.
Moreover, we need the notion of conic divisorial ideals of toric rings to compute (dual) $F$-signatures.
Thus, we first introduce toric rings and their conic divisorial ideals.

Let $\sfM\subset \ZZ^d$ be a lattice of rank $d$ and let $\sfN =\Hom_\ZZ(\sfM, \ZZ)$ be the dual lattice of $\sfM$. 
We set $\sfM_\RR =\sfM\otimes_\ZZ\RR$ and $\sfN_\RR =\sfN\otimes_\ZZ\RR$ and 
denote the natural pairing by $\langle-,-\rangle:\sfM_\RR\times\sfN_\RR\rightarrow\RR$.
We consider a strongly convex rational polyhedral cone 
$$
\tau =\cone(v_1, \cdots, v_n)=\RR_{\ge 0}v_1+\cdots +\RR_{\ge 0}v_n\subset\sfN_\RR 
$$
of dimension $d$ generated by $v_1, \cdots, v_n\in\sfN$ where $d\le n$. 
We assume this system of generators is minimal and the generators are primitive, i.e., $\epsilon v_i \notin \sfN$ for any $0<\epsilon <1$.
For each generator, we define a linear form $\sigma_i(-)\coloneqq\langle-, v_i\rangle$ and denote $\sigma(-)=(\sigma_1(-),\cdots,\sigma_n(-))$. 
We consider the dual cone $\tau^\vee$: 
$$
\tau^\vee=\{{\bf x}\in\sfM_\RR \mid \sigma_i({\bf x})\ge0 \text{ for all } i\in [n] \}, 
$$
where we let $[n]:=\{1,\ldots,n\}$.
We now define the toric ring of $\tau^\vee$ with respect to $\sfM$
\begin{align}\label{R}
R=\kk[\tau^\vee\cap\sfM]=\kk[u_1^{\alpha_1}\cdots u_d^{\alpha_d} : (\alpha_1, \cdots, \alpha_d)\in\tau^\vee\cap\sfM]. 
\end{align}
Toric rings are normal (and hence Cohen-Macaulay) domain.

\medskip

For each $\bfa=(a_1, \cdots, a_n)\in\RR^n$, we set 
$$
\TT(\bfa)=\{{\bf x}\in\sfM : \sigma_i({\bf x})\ge a_i \text{ for all } i\in [n] \}. 
$$
Then, we define the module $\calT(\bfa)$ generated by all monomials whose exponent vector is in $\mathbb{T}(\bfa)$. 
By the definition, we have $\mathbb{T}(0)=\tau^\vee\cap\sfM$ and $\calT(0)=R$.
Moreover, we note some facts associated with the module $\calT(\bfa)$ (see e.g., \cite[Section 4.F]{BG2}): 
\begin{itemize}
\item Since $\sigma_i(\bfx)\in \ZZ$ for any $i\in [n]$ and any $\bfx \in \sfM$, we can see that $\calT(\bfa)=\calT(\left\lceil \bfa\right\rceil)$, where $\left\lceil \; \right\rceil$ means the round up 
and $\left\lceil \bfa\right\rceil=(\left\lceil a_1\right\rceil, \ldots, \left\lceil a_n\right\rceil)$. 
\item The module $\calT(\bfa)$ is a divisorial ideal and any divisorial ideal of $R$ takes this form.
Therefore, we can identify each $\bfa\in\ZZ^n$ with the divisorial ideal $\calT(\bfa)$.
\item It is known that the isomorphic classes of divisorial ideals of $R$ one-to-one correspond to the elements of the divisor class group $\Cl(R)$ of $R$. 
We see that for $\bfa, \bfa^\prime\in\ZZ^n$, $\calT(\bfa)\cong \calT(\bfa^\prime)$ if and only if there exists ${\bf y}\in \sfM$ such that $a_i=a_i^\prime+\sigma_i({\bf y})$ for all $i\in [n]$. 
Thus, we have $\Cl(R)\cong\ZZ^n/\sigma(\sfM)$.
\end{itemize} 

Now, we define a special class of divisorial ideals called \textit{conic}.
\begin{defi}[{see e.g., \cite[Section~3]{BG2}}]
A divisorial ideal $\calT(\ab)$ is said to be \textit{conic} if there is $\xb \in \RR^d$ with $\ab=\left\lceil\sigma(\xb)\right\rceil$. 
In other words, there is $\xb \in \RR^d$ such that $a_i-1 < \sigma_i(\xb) \leq a_i$ for all $i\in [n]$.
We denote the set of isomorphism classes of conic divisorial ideals of $R$ by $\scrC(R)$.
\end{defi}
We have the following characterization of conic divisorial ideals:

\begin{prop}[{\cite[Corollary 1.2]{Bru}}] 
\label{conic_characterization}
Let the notation be the same as above and let 
$$L_{i,a_i}=\{{\bfx} \in \sfM_\RR \mid a_i-1<\sigma_i({\bfx}) \leq a_i\}.$$
Then any conic divisorial ideal is isomorphic to $\calT(a_1,\cdots,a_n)$ for some $(a_1,\cdots,a_n)\in\ZZ^n$ satisfying the following condition:

\begin{itemize}
\item the cell $\bigcap_{i=1}^n L_{i,a_i}$ is a full-dimensional cell of the decomposition of the semi-open cube $(-1,0]^d$ 
given by hyperplanes $\{{\bf x}\in\sfM_\RR \mid \sigma_i({\bf x})=m \}$ for some $m \in\ZZ$ and $i=1,\cdots,n$.
\end{itemize}
\end{prop} 

\bigskip

\subsection{Generalized $F$-signatures and dual $F$-signatures of toric rings}\label{subsec_signature}
In this subsection, we discuss the generalized $F$-signatures and dual $F$-signatures of toric rings.
First, we introduce the notion of finite $F$-representation type (see \cite{SmVdB,Yao}).

Let $R$ be a $d$-dimensional $F$-finite Noetherian local domain of prime characteristic $p>0$. 
We say that $R$ has \emph{finite $F$-representation type} (for short, \emph{FFRT}) by $\calS$ if there is a finite set 
$\calS\coloneqq\{M_0, M_1, \cdots, M_n\}$ of isomorphism classes of indecomposable finitely generated $R$-modules such that 
for any $e\in\ZZ_{>0}$, the $R$-module $R^{1/p^e}$ is isomorphic to a finite direct sum of these modules: 
\begin{equation}
\label{A_decomp}
R^{1/p^e}\cong M_0^{\oplus c_{0,e}}\oplus M_1^{\oplus c_{1,e}}\oplus\cdots\oplus M_n^{\oplus c_{n,e}}
\end{equation}
for some $c_{i,e}\ge 0$. 
Moreover, we say that a finite set $\calS=\{M_0, \cdots, M_n\}$ is the \emph{FFRT system} of $R$ if every $R$-module $M_i$ appears non-trivially in $R^{1/p^e}$ as a direct summand for some $e\in\ZZ_{>0}$. 

Assume that $R$ has FFRT by the FFRT system $\calS=\{M_0, M_1, \cdots, M_n\}$. 
We consider the limit 
\begin{equation}
\label{def_genFsig}
s_{\mathrm{gen}}(M_i)\coloneqq\lim_{e\rightarrow\infty}\frac{c_{i,e}}{p^{ed}} \quad (i=0,1,\cdots,n). 
\end{equation}
The existence of the limit is guaranteed under the assumption $R$ has FFRT (see \cite[Theorem~3.11]{Yao}), and we call $s_{\mathrm{gen}}(M_i)$ the \emph{generalized $F$-signature} of $M_i$.
Note that if $M_i=R$ for some $i$, then $s_{\mathrm{gen}}(R)$ is nothing but the $F$-signature of $R$.

\bigskip

Let $R$ be a toric ring as in Section~\ref{subsec_preli}.
The following theorem tells us that $R$ has FFRT by the FFRT system $\scrC(R)$:
\begin{theorem}[{\cite[Proposition~3.6]{BG1},\cite[Proposition~3.2.3]{SmVdB}}]
\label{motivation_thm}
For any $e\in\ZZ_{>0}$, the $R$-module $R^{1/p^e}$ is a direct sum of conic divisorial ideals. 
Moreover, all conic divisorial ideals appear in $R^{1/p^e}$ as direct summands for $e\gg 0$.
That is, if we set $\scrC(R)=\{M_0:=R, M_1,\cdots, M_n\}$, then for $e\in\ZZ_{>0}$, we can write 
$$R^{1/p^e}\cong R^{\oplus c_{0,e}}\oplus M_1^{\oplus c_{1,e}}\oplus\cdots\oplus M_n^{\oplus c_{n,e}}$$
for some $c_{i,e}\ge0$.
\end{theorem} 

Moreover, the combinatorial description of conic divisorial ideals given in Proposition~\ref{conic_characterization} helps us to compute the generalized $F$-signatures. 

\begin{theorem}[{\cite[Section~3]{Bru}, see also \cite{Kor,Sin,WY}}]
\label{Thm_Fsig_volume}
Let $M$ be a conic divisorial ideal of $R$. 
Then, the generalized $F$-signature of $M$ is the sum of volumes of the full-dimensional cell $\bigcap_{i=1}^n L_{i,a_i}$ 
appearing in the decomposition of $(-1,0]^d$ (see Proposition~\ref{conic_characterization}) such that $M\cong \calT(a_1,\cdots, a_n)$. 
\end{theorem}

\bigskip

We recall a formula to calculate the dual $F$-signatures of toric rings given in \cite{ST}.
Let $\Vol(\sfM)$ be the Euclidean volume of an elementary parallelepiped of the lattice $\sfM$.
It is well known that the ideal
$I_R=(t_1^{\alpha_1}\cdots t_d^{\alpha_d} : (\alpha_1,\ldots,\alpha_d)\in (\tau^{\vee})^{\circ}\cap \ZZ^d)$
is isomorphic to $\omega_R$ (\cite{S78}), where $(\tau^{\vee})^{\circ}$ denotes the relative interior of $\tau^{\vee}$.
Let $\calG(R)$ be the set consisting of the points in $(\tau^{\vee})^{\circ}\cap \sfM$ corresponding to the monomials which are the minimal generators of $I_R$.

\begin{prop}[{\cite[Proposition~6.1]{ST}}]\label{prop go toric}
Work with the same notation as above.
Then we have
\begin{align*}
s_{\mathrm{dual}}(R) = \min\set{\frac{1}{\Vol(\sfM)\cdot|T|}
\Vol \left (\bigcup_{a \in T} \tau^\vee \cap \left (a  - \tau^\vee \right) \right) : \emptyset \neq T\subset \calG(R)}.
\end{align*}
\end{prop}

\bigskip

\subsection{$O$-sequences}\label{sec_Oseq}
In this subsection, we introduce the notion of $O$-sequences and give a lemma, which is required to get our results.

We define the following partial order on $\ZZ^{n}_{>0}$; for $a,b \in \ZZ^{n}_{>0}$, we write $a\preceq b$ if $b-a\in \ZZ^{n}_{\ge 0}$.
For a non-empty subset $T\subset \ZZ^{n}_{>0}$, let $I(T):=\bigcup_{a\in T}\set{b\in \ZZ^{n}_{>0} : b\preceq a}$.
We denote by $I(a)$ instead of $I(\{a\})$ for $a\in \ZZ^{n}_{>0}$.
In addition, for $l\in \ZZ_{>0}$, we set
$$H_l(T):=\set{(x_1,\ldots,x_n)\in I(T) : x_1+\cdots+x_{n}=l}$$
and $h_l(T):=|H_{l+n}(T)|$.
Then a sequence $(h_0,h_1,\ldots,h_s)$ is said to be an \textit{$O$-sequence} if there exists $T\subset \ZZ^n$ for some $n\in \ZZ_{>0}$ such that $h_i=h_i(T)$ for each $i=0,1,\ldots,s$.

Given positive integers $f$ and $i$, there exists a unique representation, called the \textit{$i$-binomial representation} of $f$:
$$f=\binom{n_i}{i}+\binom{n_{i-1}}{i-1}+\cdots+\binom{n_{i-j}}{i-j},$$
where $n_i>n_{i-1}>\cdots>n_{i-j}\ge i-j\ge1$.
Then we define
$$f^{\langle i \rangle}:=\binom{n_i+1}{i+1}+\binom{n_{i-1}+1}{i}+\cdots+\binom{n_{i-j}+1}{i-j+1}$$
and
$$f^{(i)}:=\binom{n_i}{i+1}+\binom{n_{i-1}}{i}+\cdots+\binom{n_{i-j}}{i-j+1}.$$

The following characterization of $O$-sequences is known:
\begin{thm}[{\cite{Mac}}]\label{thm_mac}
A sequence $(h_0,h_1,\ldots,h_s)\in \ZZ^{s+1}$ is an $O$-sequence if and only if (i) $h_0=1$ and (ii) for each $i=0,\ldots,s-1$, one has $0\le h_{i+1}\le h_i^{\langle i \rangle}$.
\end{thm}
For positive integers $a$, $b$ and $i$, we have $a^{\langle i \rangle}\le b^{\langle i \rangle}$ if $a\le b$ (\cite[Lemma~4.2.13 (a)]{BH}).
Therefore, for an $O$-sequence $(h_0,h_1,\ldots,h_s)$, we can see that $h_i\le \binom{h_i+i-1}{i}$ for each $i=0,1,\ldots,s$ since the inequality $h_{i+1}\le h_i^{\langle i \rangle}$ holds.
For further information on the $O$-sequences, see, e.g., \cite[Section~2]{S78} and \cite[Chapter~4.2]{BH}.

\bigskip

The following is a key lemma to show our theorems:
\begin{lemma}\label{key_lemma}
    Let $(h_0,\ldots,h_s)$ be an $O$-sequence with $h_1\le n+1$.
    Then for any $i=0,\ldots,s$, we have 
    \begin{align}\label{lem_ineq}
    \frac{h_i}{h_s} \ge \frac{\binom{n+i}{i}}{\binom{n+s}{s}}.
    \end{align}
\end{lemma}
\begin{proof}
Note that the inequality (\ref{lem_ineq}) holds if $i=s$.
If the inequality $ h_i/h_{i+1} \ge \binom{n+i}{i}/\binom{n+i+1}{i+1}$ holds for each $0\le i \le s-1$, then we can see that
$$\frac{h_i}{h_s}=\frac{h_i}{h_{i+1}}\frac{h_{i+1}}{h_{i+2}}\cdots\frac{h_{s-1}}{h_s}\ge \frac{\binom{n+i}{i}}{\binom{n+i+1}{i+1}}\frac{\binom{n+i+1}{i+1}}{\binom{n+i+2}{i+2}}\cdots\frac{\binom{n+s-1}{s-1}}{\binom{n+s}{s}}=\frac{\binom{n+i}{i}}{\binom{n+s}{s}}.$$
Thus, it is enough to show that $h_{s-1}/h_s \ge \binom{n+s-1}{s-1}/\binom{n+s}{s}$, which is equivalent to $(n+s)h_{s-1}\ge sh_s$.

Since $h_s\le h_{s-1}^{\langle s-1 \rangle}$ from Theorem~\ref{thm_mac}, 
if the inequality $(n+s)h_{s-1}\ge sh_{s-1}^{\langle s-1 \rangle}$ (this is equivalent to $nh_{s-1}\ge sh_{s-1}^{(s-1)}$) holds, then we get the desired result; $(n+s)h_{s-1}\ge sh_{s-1}^{\langle s-1 \rangle}\ge sh_s$.
Therefore, it suffices to prove that $nh_{s-1}\ge sh_{s-1}^{(s-1)}$.

Note that $h_{s-1}\le \binom{h_1+s-2}{s-1}\le \binom{n+s-1}{s-1}$.
The equality $nh_{s-1}=sh_{s-1}^{(s-1)}$ holds trivially if $h_{s-1}=\binom{n+s-1}{s-1}$, so we may assume that $h_{s-1}<\binom{n+s-1}{s-1}$ and $h_{s-1}$ has the following binomial representation:
$$h_{s-1}=\binom{n_{s-1}}{s-1}+\binom{n_{s-2}}{s-2}+\cdots +\binom{n_{s-j}}{s-j}.$$
Note that $n_{s-1}<n+s-1$.
We can see that for each $p=1,\ldots,j$,
\begin{align*}
s\binom{n_{s-p}}{s-p+1}&=((s-p+1)+(p-1))\binom{n_{s-p}}{s-p+1} \\
&=(n_{s-p}-s+p)\binom{n_{s-p}}{s-p}+(p-1)\binom{n_{s-p}}{s-p+1},
\end{align*}
and hence 
\begin{align*}
nh_{s-1}-sh_{s-1}^{(s-1)}&=n\sum_{p=1}^j\binom{n_{s-p}}{s-p}-s\sum_{p=1}^j\binom{n_{s-p}}{s-p+1} \\
&=\sum_{p=1}^j (n+s-p-n_{s-p})\binom{n_{s-p}}{s-p}-\sum_{p=1}^j(p-1)\binom{n_{s-p}}{s-p+1} \\
&=\underbrace{\sum_{p=1}^j (n+s-p-1-n_{s-p})\binom{n_{s-p}}{s-p}}_{(a)} \\
& \quad \quad \quad \quad \quad \quad \quad \quad \quad+\underbrace{\sum_{p=1}^j\rbra{\binom{n_{s-p}}{s-p}-(p-1)\binom{n_{s-p}}{s-p+1}}}_{(b)}.
\end{align*}
We show that $(a)\ge 0$ and $(b)\ge 0$.
It follows from $n+s-1>n_{s-1}>n_{s-2}>\cdots>n_{s-j}$ that $n+s-p-1-n_{s-p}\ge 0$ for each $p$.
Thus, we have $(a)\ge 0$.

We can see that 
\begin{align*}
\binom{n_{s-p}}{s-p}&=\binom{n_{s-p}-1}{s-p}+\binom{n_{s-p}-1}{s-p-1}=\binom{n_{s-p}-1}{s-p}+\rbra{\binom{n_{s-p}-2}{s-p-1}+\binom{n_{s-p}-2}{s-p-2}} \\
&=\cdots=\rbra{\sum_{i=1}^{j-p}\binom{n_{s-p}-i}{s-p-i+1}}+\binom{n_{s-p}-(j-p)}{s-j}
\end{align*}
and 
$$\sum_{p=1}^{j-1}\binom{n_{s-p}}{s-p}=\sum_{p=1}^{j-1}\rbra{\sum_{q=1}^p\binom{n_{s-q}+q-p-1}{s-p}}+\sum_{p=1}^{j-1}\binom{n_{s-p}-(j-p)}{s-j}.$$
Therefore, we have
\begin{align*}
(b)&=\sum_{p=1}^{j-1}\rbra{\binom{n_{s-p}}{s-p}-p\binom{n_{s-p-1}}{s-p}}+\binom{n_{s-j}}{s-j} \\
&=\sum_{p=1}^{j-1}\rbra{\sum_{q=1}^p\rbra{\binom{n_{s-q}+q-p-1}{s-p}-\binom{n_{s-p-1}}{s-p}}}+\sum_{p=1}^{j}\binom{n_{s-p}-(j-p)}{s-j}.
\end{align*}
By $n_{s-q}+q-p-1\ge n_{s-p-1}$, the inequality $(b)\ge 0$ holds and we conclude that $nh_{s-1}\ge sh_{s-1}^{(s-1)}$.
\end{proof}

\bigskip

\section{Dual $F$-signatures of Veronese subrings of polynomial rings}\label{sec_vero}
This section is devoted to calculating the dual $F$-signatures of Veronese subrings of polynomial rings.

Let $V_{n,d}$ be the $n$th Veronese subring of the polynomial ring $S=\kk[x_1,\ldots,x_d]$, i.e., 
$$V_{n,d}:=\bigoplus_{i\ge 0}S_{in}.$$
We can realize $V_{n,d}$ as a toric ring; let $\sfM:=\set{(x_1,\ldots,x_d)\in \ZZ^d : x_1+\cdots+x_d \equiv 0 \; \mod n}$ and 
let $\tau^{\vee}=\set{(x_1,\ldots,x_d)\in \RR^d : x_i\ge 0 \text{ for all }i\in [d]}$.
Then we have $V_{n,d}=\kk[\tau^{\vee}\cap \sfM]$ and 
$$\calG(V_{n,d})=\set{(a_1,\ldots,a_d)\in\ZZ_{>0}^d : a_1+\cdots+a_d=n\left \lceil \frac{d}{n} \right \rceil}.$$

\bigskip

\begin{thm}\label{dual_ver}
We have 
$s_{\mathrm{dual}}(V_{n,d})=\frac{1}{d}\left\lceil\frac{d}{n} \right\rceil$.
\end{thm}
\begin{proof}
We compute $s_{\mathrm{dual}}(V_{n,d})$ using Proposition~\ref{prop go toric}.
Take a subset $\emptyset \neq T\subset \calG(V_{n,d})$.
Here, we observe the following (compare with \cite[Example~6.2]{ST});
the shape of $\bigcup_{a \in T} \tau^\vee \cap \left (a  - \tau^\vee \right)$ looks like a ``building block pyramid".
More precisely, it can be divided into integral unit cubes, and each cube is identified with its vertex with the largest coordinates, that is, these cubes one-to-one correspond to the elements of $I(T)$.
Since the volume of the unit cube is $1$, we have 
\begin{align*}
\Vol \left (\bigcup_{a \in T} \tau^\vee \cap \left (a  - \tau^\vee \right) \right)=|I(T)|=\sum_{l=d}^{n\left\lceil\frac{d}{n} \right\rceil}|H_l(T)|=h_0(T)+\cdots+h_s(T),
\end{align*}
where $s=n\left\lceil\frac{d}{n} \right\rceil-d$. Note that $\Vol(\sfM)=n$ and $|T|=h_s(T)$.
Thus, it follows from Lemma~\ref{key_lemma} that
\begin{align}\label{ineq_vero}
\frac{1}{\Vol(\sfM)\cdot |T|}
\Vol \left (\bigcup_{a \in T} \tau^\vee \cap \left (a  - \tau^\vee \right) \right)&=\frac{1}{n}\sum_{l=0}^{s}\frac{h_l(T)}{h_s(T)} \notag \\
&\ge \frac{1}{n}\sum_{l=0}^s\frac{\binom{d+l-1}{l}}{\binom{d+s-1}{s}}=\frac{1}{n}\frac{\binom{d+s}{s}}{\binom{d+s-1}{s}}=\frac{d+s}{nd}=\frac{1}{d}\left\lceil\frac{d}{n} \right\rceil.
\end{align}
Moreover, if $T=\calG(V_{n,d})$, then we have $h_l(T)=\binom{d+l-1}{l}$ for $l=0,\ldots,s$, and hence the equality of (\ref{ineq_vero}) holds.
Therefore, it follows from Proposition~\ref{prop go toric} that
$$s_{\mathrm{dual}}(V_{n,d}) = \min\set{\frac{1}{\Vol(\sfM)\cdot|T|}
\Vol \left (\bigcup_{a \in T} \tau^\vee \cap \left (a  - \tau^\vee \right) \right) : \emptyset \neq T\subset \calG(V_{n,d})}=\frac{1}{d}\left\lceil\frac{d}{n} \right\rceil.$$
\end{proof}

\bigskip

\section{Dual $F$-signatures of Segre products of polynomial rings}\label{sec_segre}
In this section, we discuss the dual $F$-signatures of Segre products of polynomial rings.

\subsection{Segre products of polynomial rings and their generalized $F$-signatures}
To compute the dual $F$-signature of the Segre product of polynomial rings, we need some basic properties about it and its generalized $F$-signature.
Thus, we recall them in this subsection.

Let $r_1,\ldots,r_t$ be positive integers with $r_1\le \cdots \le r_t$ and let $S_{r_1,\ldots,r_t}$ be the Segre product $\#_{i=1}^t \kk[x_{i,1},\cdots,x_{i,r_i+1}]$ of polynomial rings, 
which is the ring generated by monomials with the form $x_{1,j_1}x_{2,j_2}\cdots x_{t,j_t}$ where $i\in [t]$ and $j_i\in [r_i+1]$.

We can describe $S_{r_1,\ldots,r_t}$ as a toric ring as follows;
we consider the toric ring $R:=\kk[\rho^{\vee}\cap\ZZ^d]$ of the cone $\rho^{\vee}\subset \RR^d$ with respect to $\ZZ^d$ where $d=r_1+\cdots+r_t+1$ and we define $\rho^{\vee}$ as follows:
\begin{align}\label{cone}
\rho^{\vee}:=\set{\yb:=(y_{1,1},\ldots,y_{1,r_1},y_{2,1},\ldots,y_{t,r_t},y_d)\in \RR^d :
\hspace{-0.4cm}
\begin{array}{rl}
    &y_{i,j_i}\ge 0 \text{ for $i\in[t]$ and $j_i\in [r_i]$}, \vspace{0.2cm}\\
     &y_d-\sum_{j=1}^{r_i}y_{i,j}\ge 0 \text{ for $i\in [t]$}
    \end{array}
}.
\end{align}
Then we can see that $R$ is generated by the monomials $u_{1,j_1}\cdots u_{t,j_t}u_d$ ($i\in[t]$ and $j_i\in[r_i+1]$) where we let $u_{i,r_i+1}=1$ for $i\in [t]$, and hence $R$ is standard graded by setting $\deg(u_{1,j_1}\cdots u_{t,j_t}u_d)=1$.
Moreover, we have $S_{r_1,\ldots,r_t}\cong R$.
In what follows, we identify these rings.

\bigskip

Next, we determine $\calG(R)$.
For $i\in [t-1]$, let $$\calF_i:=\set{\kb_i:=(k_{i,1},\ldots,k_{i,r_i+1})\in\ZZ_{>0}^{r_i+1} : |\kb_i|:=k_{i,1}+\cdots+k_{i,r_i+1}=r_t+1}$$ and let $\calF:=\calF_1\times\cdots\times\calF_{t-1}$. 
Moreover, for $T\subset \calF$ and $i\in [t-1]$, let $T_i$ be the image of the $i$th projection $\calF \to \calF_i$. 

We can see that the Segre product of polynomial rings is \textit{level}, that is, all the degrees of the minimal generators of its canonical module are the same.
From the inequalities in (\ref{cone}), all the monomials corresponding to the points in $(\rho^{\vee})^{\circ}\cap\ZZ^d$ have at least $r_t+1$ degree.
Thus, we have
\begin{multline*}
\calG(R)=\Big\{\widetilde{\kb}:=(k_{1,1},\ldots,k_{1,r_1},k_{2,1},\ldots,k_{t-1,r_{t-1}},\underbrace{1,\ldots,1}_{r_t},r_t+1) \in \RR^d : \\
\kb:=(\kb_1,\kb_2,\ldots,\kb_{t-1})\in \calF\Big\}.
\end{multline*}
Notice that the map $\kb \mapsto \widetilde{\kb}$ is a bijection between $\calF$ and $\calG(R)$, so we identify these two sets.

\bigskip

Finally, we recall the conic divisorial ideals of $R$ and their generalized $F$-signatures, discussed in \cite{HN,HN2}.
We pick up and summarize the results obtained in these papers \cite{HN,HN2}, which we will need later.

\begin{itemize}
\item (\cite[Example~2.6]{HN}) In this case, $\Cl(R)$ is isomorphic to $\ZZ^{t-1}$. 
We set 
\begin{align}\label{segre_ineq}
\calC:=\set{(y_1,\cdots,y_{t-1}) \in \ZZ^{t-1} : \hspace{-0.4cm}
\begin{array}{ll}
&-r_t \leq y_i \leq r_i \text{ for } i\in [t-1], \vspace{0.2cm}\\
&-r_j \leq y_i-y_j \leq r_i \text{ for }i,j\in [t-1]
\end{array}
}.
\end{align}
Let $M_\zb$ be a divisorial ideal corresponding to $\zb=(z_1,\cdots,z_{t-1})\in\Cl(R)$. 
Then we see that $M_\zb$ is conic if and only if $\zb\in\calC$. 

\item (\cite[Corollary~4.7 and Example~5.1]{HN2}) The generalised $F$-signature of $M_\zb$ for $\zb\in \calC$ coincides with the volume of the following polytope:
\begin{align}\label{region_F}
\set{\yb \in \RR^{d} : 
 \hspace{-0.4cm}
 \begin{array}{ccc}
     & -1\le y_{i,j}\le 0 \text{ for } i\in [t] \text{ and }j\in [r_i] \\
     & -1 \le y_d\le 0 \\
     & z_i-1 \le y+\sum_{j=1}^{r_t}y_{t,j}- \sum_{j'=1}^{r_i}y_{i,j'}\le z_i \text{ for } i\in [t-1] 
\end{array}}
\end{align}
In particular, if $\zb\in \Cl(R)$ does not belong to $\calC$, this polytope is not full-dimensional, and hence its volume is equal to $0$ and we let $s_{\mathrm{gen}}(M_\zb)=0$. 
\item The generalized $F$-signatures of the Segre products of polynomial rings are computed by counting the elements in the symmetric group satisfying certain conditions (\cite[Theorem~5.3]{HN2}).
Especially, in the case of the Segre product of two polynomial rings, they are explicitly given as follows:
\end{itemize}

\begin{prop}[{\cite[Theorem~6.1]{HN2}, see also \cite[Theorem~5.8]{WY}, \cite[Example 7]{Sin}}]
\label{compute_Fsig2}
Suppose that $t=2$.
For $\zb=z_1\in \calC\subset \ZZ$,
we have $s_{\mathrm{gen}}(M_\zb)=A_{z_1+r_1,d}/d!$.
\end{prop}
\noindent Here, $A_{k,n}$ denotes the \textit{Eulerian number}, which is the number of permutations of the numbers $1$ to $n$ in which exactly $k$ elements are less than the previous element. 

\bigskip

\subsection{Computations of $s_{\mathrm{dual}}(S_{r_1,\ldots,r_n})$}
We explain how to calculate $s_{\mathrm{dual}}(S_{r_1,\ldots,r_n})$ using Proposition~\ref{prop go toric}. 

For $\kb\in \calF$, we can see that $P_\kb:=\rho^{\vee}\cap (\widetilde{\kb}-\rho^{\vee})$ has the following representation:
$$P_\kb=\set{\yb \in \RR^{d} : 
\hspace{-0.4cm}
\begin{array}{cccc}
    & 0 \le y_{i,j}\le k_{i,j} \text{ for } i\in [t-1] \text{ and }j\in [r_i] \\
    & 0 \le y_d - (y_{i,1}+\cdots+y_{i,r_i})\le k_{i,r_i+1} \text{ for } i\in [t-1] \\
    & 0 \le y_{t,j}\le 1 \text{ for } j\in [r_i] \\
    & 0 \le y_d - (y_{t,1}+\cdots+y_{t,r_t})\le 1
\end{array}
}.
$$

To calculate the volume of $\bigcup_{\kb\in T}P_\kb$ for $T\subset \calF$, we want to consider it as a ``building block pyramid", as in Veronese subrings.
To this end, we slice $P_\kb$ as follows:
$$P_\kb=\bigcup\set{Q_\cb : \cb=(\cb_1,\ldots,\cb_{t-1})\in I(\kb)},$$
where we set
\begin{align}\label{poly_q}
    Q_\cb:=\set{\yb \in \RR^{d} : 
\hspace{-0.4cm}
\begin{array}{cccc}
    & c_{i,j}-1 \le y_{i,j}\le c_{i,j} \text{ for } i\in [t-1] \text{ and }j\in [r_i] \\
    & c_{i,r_i+1}-1 \le y_d - (y_{i,1}+\cdots+y_{i,r_i})\le c_{i,r_i+1} \text{ for } i\in [t-1] \\
    & 0 \le y_{t,j}\le 1 \text{ for } j\in [r_i] \\
    & 0 \le y_d - (y_{t,1}+\cdots+y_{t,r_t})\le 1
\end{array}
}.
\end{align}
The pieces $Q_{\cb}$ appearing in $\bigcup_{\kb\in T}P_\kb$ are parameterized by $I(T)$, that is, $\bigcup_{\kb\in T}P_\kb=\bigcup_{\cb\in I(T)}Q_\cb$.
In particular, we have $\Vol(\bigcup_{\kb\in T}P_\kb)=\sum_{\cb\in I(T)}\Vol(Q_\cb)$.

\begin{lemma}\label{lem_genF}
For $\cb=(\cb_1,\ldots,\cb_{t-1})\in I(\kb)$, the volume of $Q_\cb$ coincides with $s_{\mathrm{dual}}(M_\zb)$ where $\zb:=(|\cb_1|-r_t-1,\ldots,|\cb_{t-1}|-r_t-1)$.
\end{lemma}
\begin{proof}
Apply the following unimodular transformation: 
\begin{align*}
&y_{i,j} \mapsto w_{i,j}+c_{i,j} \;\text{ for }i\in [t-1] \text{ and } j\in [r_i], \\
&y_{t,j} \mapsto w_{t,j}+1 \;\text{ for } j\in [r_t], \\
&y_d \mapsto w_d+w_{t,1}+\cdots+w_{t,r_t}+r_t+1.
\end{align*}
Then $Q_{\cb_1,\ldots,\cb_{t-1}}$ changes as follows: 
\begin{align*}
\set{\yb \in \RR^{d} : 
 \hspace{-0.4cm}
 \begin{array}{ccc}
     & -1\le y_{i,j}\le 0 \text{ for } i\in [t] \text{ and }j\in [r_i] \\
     & -1 \le y_d\le 0 \\
     & |\cb_i|-r_t-2 \le y+\sum_{j=1}^{r_t}y_{t,j}- \sum_{j'=1}^{r_i}y_{i,j'}\le |\cb_i|-r_t-1 \text{ for } i\in [t-1] 
\end{array}}.
\end{align*}
By comparing it with (\ref{region_F}), we can see that the assertion holds.
\end{proof}

We set
\begin{align*}
    \calD&:=\set{(y_1,\cdots,y_{t-1}) \in \ZZ^{t-1} : \hspace{-0.4cm}
\begin{array}{ll}
&-r_t+r_i \leq y_i \leq 0 \text{ for } i\in [t-1], \vspace{0.2cm}\\
&-r_j \leq y_i-y_j \leq r_i \text{ for }i,j\in [t-1]
\end{array}} \text{ and } \\
&\; \\
    \calE&:=\set{\lb=(l_1,\ldots,l_{t-1})\in \ZZ^{t-1} : r_i+1\le l_i \le r_t+1 \text{ for all }i\in [t-1]}.
\end{align*}

\bigskip

We reach the main result of this section.
\begin{thm}\label{dual_segre}
Let $r_1,\ldots,r_t$ be positive integers with $r_1\le \cdots \le r_t$.
Then we have
\begin{align}\label{ineq_dual2}
s_{\mathrm{dual}}(S_{r_1,\ldots,r_t})\le \sum_{\zb\in \calD}\rbra{\prod_{i=1}^{t-1}\frac{\binom{r_t+z_i}{r_i}}{\binom{r_t}{r_i}}}s_{\mathrm{gen}}(M_{\zb}).
\end{align}
\end{thm}

\begin{proof}
From the above observation, for $\emptyset \neq T\subset \calF$, we obtain 
\begin{align*}
\Vol \left (\bigcup_{\kb \in T} \rho^\vee \cap \left (\widetilde{\kb}  - \rho^\vee \right) \right)=\Vol\rbra{\bigcup_{\kb\in T}P_\kb}=\sum_{\cb\in I(T)}\Vol\rbra{Q_\cb}
\end{align*}
Now, we let $T=\calF$. Then we have
\begin{align}\label{gap}
    I(T)&=I(T_i)\times\cdots\times I(T_{t-1}) \\
    &=\bigcup_{\lb\in \calE}H_{l_1}(\calF_1)\times\cdots\times H_{l_{t-1}}(\calF_{t-1}). \notag
\end{align}
Thus,
\begin{align*}
\sum_{\cb\in I(\calF)}\Vol\rbra{Q_\cb}=\sum_{\lb\in \calE}\rbra{\sum_{\cb\in H_{l_1}(\calF_1)\times\cdots\times H_{l_{t-1}}(\calF_{t-1})}\Vol\rbra{Q_\cb}}.
\end{align*}
It follows from Lemma~\ref{lem_genF} that $\Vol\rbra{Q_\cb}$ for $\cb\in H_{l_1}(\calF_1)\times\cdots\times H_{l_{t-1}}(\calF_{t-1})$ is equal to $s_{\mathrm{gen}}(M_\zb)$ where $\zb=(z_1,\ldots,z_{t-1}):=(l_1-r_t-1,\ldots,l_{t-1}-r_t-1)$.
Hence, we have
\begin{align*}
\frac{1}{|T|}\sum_{\cb\in I(\calF)}\Vol\rbra{Q_\cb}&=\frac{1}{\prod_{i=1}^{t-1}\binom{r_t}{r_i}}\sum_{\lb\in\calE}\rbra{\prod_{i=1}^{t-1}\binom{l_i-1}{r_i}}s_{\mathrm{gen}}(M_\zb) \\
&=\sum_{\zb\in \calD}\rbra{\prod_{i=1}^{t-1}\frac{\binom{r_t+z_i}{r_i}}{\binom{r_t}{r_i}}}s_{\mathrm{gen}}(M_{\zb}).
\end{align*}
Therefore, by Proposition~\ref{prop go toric}, we get
\begin{align*}
s_{\mathrm{dual}}(S_{r_1,\ldots,r_t}) &= \min\set{\frac{1}{\Vol(\ZZ^d)\cdot|T|}
\Vol \left (\bigcup_{\kb \in T} \rho^\vee \cap \left (\widetilde{\kb}  - \rho^\vee \right) \right) : \emptyset \neq T\subset \calF}\\
&\le 
\sum_{\zb\in \calD}\rbra{\prod_{i=1}^{t-1}\frac{\binom{r_t+z_i}{r_i}}{\binom{r_t}{r_i}}}s_{\mathrm{gen}}(M_{\zb}).
\end{align*}
\end{proof}

We can show that if the equality (\ref{gap}) holds for any $T\subset \calF$, then the equality of (\ref{ineq_dual2}) also holds.
In fact, when (\ref{gap}) is satisfied, we have
\begin{align*}
\frac{1}{|T|}\sum_{\cb\in I(T)}\Vol\rbra{Q_\cb}&=\sum_{\lb\in \calE}\rbra{\sum_{\cb\in H_{l_1}(T_1)\times\cdots\times H_{l_{t-1}}(T_{t-1})}\Vol\rbra{Q_\cb}} \\
&=\sum_{\lb\in \calE}\frac{\prod_{i=1}^{t-1}h_{li-r_i-1}(T_i)}{|T|}\Vol\rbra{Q_\cb}
\end{align*}
and $|T|=|T_1|\cdots|T_{t-1}|=\prod_{i=1}^{t-1}h_{r_t-r_i}(T_i)$.
Therefore, we obtain the equality of (\ref{ineq_dual2}) by applying Lemma~\ref{key_lemma}.
Unfortunately, the equality of (\ref{gap}) does not necessarily hold in general, but it does hold in the following easy case:

\begin{prop}\label{attain1}
The equality of (\ref{ineq_dual2}) holds if $r_2=\cdots=r_t$.
In particular, we have 
$$s_{\mathrm{dual}}(S_{r_1,r_2})=\frac{\sum_{l=r_1}^{r_2}\binom{l}{r_1}A_{l,d}}{\binom{r_2}{r_1}d!}.$$
\end{prop}
\begin{proof}
In this situation, we can see easily that $I(T)=I(T_1)\times\cdots\times I(T_{t-1})$ for any $\emptyset\neq T\subset \calF$ since $\calF_2=\cdots=\calF_{t-1}=\set{(1,\ldots,1)}$.
Thus, we get the desired result from the above discussion.
Especially, we get the last statement immediately from Proposition~\ref{compute_Fsig2}.
\end{proof}

Furthermore, we can also ensure that equality of (\ref{ineq_dual2}) is satisfied in the following simple case:

\begin{prop}\label{attain2}
The equality of (\ref{ineq_dual2}) holds if $r_i\in \set{r_1,r_1+1}$ for any $i\in [t]$.
\end{prop}

\begin{proof}
Assume that $r_1=\cdots=r_i$ and $r_{i+1}=\cdots=r_t=r_1+1$.
For $T\subset \calF$, let $T':=T_1\times\cdots\times T_{t-1}$ and note that $I(T')=I(T'_1)\times\cdots\times I(T'_{t-1})$ and $|T'|=|T_1|\cdots|T_i|$.
In addition, for $\lb\in \calE$, let $H_\lb(T):=\set{\cb\in I(T) : \cb_i\in H_{l_i}(T_i) \text{ for all }i\in [t-1]}$.
Then we can see that 
\begin{align*}
\frac{1}{|T|}\sum_{\cb\in I(\calF)}\Vol\rbra{Q_\cb}=\frac{1}{|T|}\sum_{\lb\in \calE}\rbra{\sum_{\cb\in H_{\lb}(T)}\Vol\rbra{Q_\cb}}
=\sum_{\lb\in \calE}\frac{|H_\lb(T)|}{|T|}s_{\mathrm{gen}}(M_\zb),
\end{align*}
where $\zb=(l_1-r_t-1,\cdots,l_{t-1}-r_t-1)$.

From the above discussion, it is enough to show that $|H_\lb(T)|/|T|\ge |H_\lb(T')|/|T'|$ for any $\lb\in \calE$.
We may assume that $l_1=\cdots=l_j=r_1+1$ and $l_j=\cdots=l_{t-1}=r_1+2$ where $0\le j\le i$.
In this situation, we can see that
$$H_\lb(T)=\set{(1,\ldots,1,\cb_{j+1},\ldots,\cb_i,1,\ldots, 1) : \cb_k \in H_{l_k}(T_k)=T_k \text{ for all }k=j+1,\ldots,i}.$$
Moreover, let $T_{j+1,\ldots,i}$ be the image of the projection $\calF \to \calF_{j+1}\times\cdots\times \calF_{i}$, then we have $|T_{j+1,\ldots,i}|=|H_\lb(T)|$, $|T_{j+1}|\cdots|T_i|=|H_\lb(T')|$ and $|T_1|\cdots|T_j||T_{j+1,\ldots,i}|=|T_1\times\cdots\times T_j\times T_{j+1,\ldots,i}|\ge |T|$.
Therefore, we obtain
$$\frac{|H_\lb(T)|}{|T|}\ge \frac{|T_{j+1,\ldots,i}|}{|T_1|\cdots|T_{j}||T_{j+1,\ldots,i}|}=\frac{|T_{j+1}|\cdots|T_i|}{|T_1|\cdots|T_i|}=\frac{|H_\lb(T')|}{|T'|}.$$
\end{proof}

\begin{ex}
We present specific values of the (dual) $F$-signatures of $S_{r_1,\ldots,r_t}$ for small $r_1,\ldots,r_t$ satisfying the conditions of Proposition~\ref{attain1} or \ref{attain2}.
We obtain them by calculating the generalized $F$-signatures (they are equal to the volume of polytopes (\ref{region_F}) and we compute them by using {\tt MAGMA} (\cite{magma})).
\begin{center}

\medskip

\begin{tabular}{ccc}
\hline \addlinespace[10pt]
 $(r_1,r_2,\ldots,r_t)$ & $s_{\mathrm{dual}}(S_{r_1,\ldots,r_t})$ & $s(S_{r_1,\ldots,r_t})$  \\ \addlinespace[10pt] \hline 
\addlinespace[10pt]
\begin{tabular}{c}
$(1,1)$ \\ \addlinespace[10pt]
$(1,2)$ \\ \addlinespace[10pt]
$(1,3)$ \\ \addlinespace[10pt]
$(2,2)$ \\ \addlinespace[10pt]
$(2,3)$ \\ \addlinespace[10pt]
$(3,3)$ \\ \addlinespace[10pt]
$(1,1,1)$ \\ \addlinespace[10pt]
$(1,1,2)$ \\ \addlinespace[10pt]
$(1,2,2)$ \\ \addlinespace[10pt]
$(2,2,2)$ \\ \addlinespace[10pt]
$(1,1,1,1)$ \\ \addlinespace[10pt]
$(1,1,1,2)$ \\ \addlinespace[10pt]
$(1,1,2,2)$ \\ \addlinespace[10pt]
$(1,2,2,2)$ \\ \addlinespace[10pt]
\end{tabular}&
\begin{tabular}{c}
$2/3$ \\ \addlinespace[10pt]
$11/16$ \\ \addlinespace[10pt]
$59/90$ \\ \addlinespace[10pt]
$11/20$ \\ \addlinespace[10pt]
$151/270$ \\ \addlinespace[10pt]
$151/315$ \\ \addlinespace[10pt]
$1/2$ \\ \addlinespace[10pt]
$41/80$ \\ \addlinespace[10pt]
$2/5$ \\ \addlinespace[10pt]
$12/35$ \\ \addlinespace[10pt]
$2/5$ \\ \addlinespace[10pt]
$129/320$ \\ \addlinespace[10pt]
$171/560$ \\ \addlinespace[10pt]
$1137/4480$ \\ \addlinespace[10pt]
\end{tabular}&
\begin{tabular}{c}
$2/3$ \\ \addlinespace[10pt]
$11/24$ \\ \addlinespace[10pt]
$13/60$ \\ \addlinespace[10pt]
$11/20$ \\ \addlinespace[10pt]
$151/360$ \\ \addlinespace[10pt]
$151/315$ \\ \addlinespace[10pt]
$1/2$ \\ \addlinespace[10pt]
$19/60$ \\ \addlinespace[10pt]
$4/15$ \\ \addlinespace[10pt]
$12/35$ \\ \addlinespace[10pt]
$2/5$ \\ \addlinespace[10pt]
$29/120$ \\ \addlinespace[10pt]
$5/28$ \\ \addlinespace[10pt]
$379/2240$ \\ \addlinespace[10pt]
\end{tabular} \\
\addlinespace[4pt]
\hline
\end{tabular}
\end{center}
\end{ex}

\bigskip

Our results would provide a sufficient reason to give the following conjecture:

\begin{conj}
Let $r_1,\ldots,r_t$ be positive integers with $r_1\le \cdots \le r_t$.
Then we have
$$s_{\mathrm{dual}}(S_{r_1,\ldots,r_t})= \sum_{\zb\in \calD}\rbra{\prod_{i=1}^{t-1}\frac{\binom{r_t+z_i}{r_i}}{\binom{r_t}{r_i}}}s_{\mathrm{gen}}(M_{\zb}).$$
\end{conj}

\end{document}